\include{birkart.sty}

\documentclass{amsart}
\usepackage{amsmath}
\usepackage{amssymb}
\usepackage{amsfonts}
\usepackage{graphicx}
\usepackage{texdraw}
\usepackage{graphpap}
\usepackage{wasysym}
\usepackage{pxfonts}
\usepackage[OT2,T1]{fontenc}

\DeclareSymbolFont{cyrletters}{OT2}{wncyr}{m}{n}
\DeclareMathSymbol{\berd}{\beta}{cyrletters}{"42}

\input txdtools

\newtheorem{thm}{Theorem}[section]
\newtheorem{cor}[thm]{Corollary}
\newtheorem{lem}[thm]{Lemma}
\newtheorem*{lemA1}{Lemma A.1}
\newtheorem*{lemA2}{Lemma A.2}
\newtheorem{prop}[thm]{Proposition}

\theoremstyle{definition}

\theoremstyle{remark}
\newtheorem{rem}[thm]{Remark}
\newtheorem*{claim}{Claim}

\numberwithin{equation}{section}
\numberwithin{figure}{section}

\newcommand{\normalsp}{{\mathcal N}}

\newcommand{\Prob}{{\mathbf{P}}}

\newcommand{\Tr}{\operatorname{Tr}}
\newcommand{\CK}{{\mathbf K}}
\newcommand{\PP}{{\mathcal P}}
\newcommand{\bs}{{\bigskip}}
\newcommand{\no}{{\noindent}}
\newcommand{\pa}{{\partial}}
\newcommand{\NN}{{\mathcal N}}
\newcommand{\E}{{\mathbf{E}}}
\newcommand{\ti}{\tilde}
\def\Int{\text{\rm int}}
\def\sm{\setminus}
\def\ms{\medskip}
\def\const{\text{\rm const.}}
\def\ss{\smallskip}
\newcommand{\bp}{{\bar\partial}}
\newcommand{\K}{{\mathbf{K}}}

\renewcommand{\d}{{\partial}}

\newcommand{\e}{\mathrm e}

\newcommand{\C}{{\mathbb C}}
\newcommand{\D}{{\mathbb D}}
\newcommand{\T}{{\mathbb T}}
\newcommand{\R}{{\mathbb R}}

\newcommand{\1}{{\mathbf 1}}

\newcommand\coity{{\mathcal C \sb 0 \sp \infty}}

\newcommand{\re}{\operatorname{Re}}
\newcommand{\im}{\operatorname{Im}}

\newcommand{\imag}{{\mathrm i}}

\newcommand{\dist}{\operatorname{dist}}
\newcommand{\supp}{\operatorname{supp}}

\newcommand{\eps}{\varepsilon}

\newcommand{\trace}{\operatorname{trace}}

\def\lpar{\left (}
\def\rpar{\right )}
\def\labs{\left |}
\def\rabs{\right |}
\def\babs#1{\labs {#1} \rabs}

\begin{document}

\title{Random normal matrices and Ward identities}

\author{Yacin~Ameur}

\address{Yacin Ameur\\
Centre for Mathematical Sciences\\
  Lund University, Sweden}

\email{Yacin.Ameur@maths.lth.se}

\author{Haakan Hedenmalm}
\address{Haakan Hedenmalm\\
Department of Mathematics\\ Royal Institute of Technology\\ S–100 44 Stockholm, Sweden}
\email{haakanh@math.kth.se}

\author{Nikolai~Makarov}

\address{Makarov: Mathematics\\
California Institute of Technology\\
Pasadena, CA 91125\\
USA}

\email{makarov@caltech.edu}


\begin{abstract} Consider the random normal matrix ensemble associated with a potential on the plane
which is sufficiently strong near infinity. It is known that, to a first approximation,
the eigenvalues obey a certain equilibrium distribution, given by Frostman's solution to the minimum energy problem of
weighted logarithmic potential theory. On a finer scale, one can consider fluctuations of eigenvalues about the equilibrium.
In the present paper, we give the correction to the expectation of fluctuations, and we
prove that the potential field of the corrected fluctuations converge on smooth test functions to a Gaussian free field with
free boundary conditions on the droplet associated with the potential.
\end{abstract}

\maketitle

Given a suitable real "weight function'' in the plane, it is well-known how to associate a corresponding (weighted)
random normal matrix ensemble (in short: RNM-ensemble). Under reasonable conditions on the weight function, the eigenvalues
of matrices picked randomly from the ensemble
will condensate on a certain compact subset $S$ of the complex plane, as the order of the matrices
tends to infinity. The set $S$ is known as the \textit{droplet} corresponding to the ensemble. It is well-known that the droplet can be described
using weighted logarithmic potential theory and, in its turn, the droplet determines the
classical equilibrium distribution of the eigenvalues (Frostman's equilibrium measure).

In this paper we prove a formula for the expectation of fluctuations about the equilibrium distribution,
for linear statistics of the eigenvalues of random normal matrices. We also prove the convergence of the
potential fields corresponding to corrected fluctuations to a Gaussian free field on $S$ with free boundary conditions.

Our approach uses \textit{Ward identities}, that is, identities satisfied by the joint intensities of the point-process
of eigenvalues, which follow from the reparametrization
invariance of the partition function of the ensemble. Ward identities are well known in field theories. Analogous results in random Hermitian matrix theory are known
due to Johansson \cite{J}, in the case of a polynomial
weight.

\subsection*{General notation} By $D(a,r)$ we mean the open Euclidean disk with center $a$ and radius $r$.
By "$\dist$'' we mean the
Euclidean distance in the plane. If $A_n$ and $B_n$ are expressions depending on a positive integer $n$, we write
$A_n\lesssim B_n$ to indicate that $A_n\le CB_n$ for all $n$ large enough where $C$ is independent of $n$. The notation
$A_n\asymp B_n$ means that $A_n\lesssim B_n$ and $B_n\lesssim A_n$.
When $\mu$ is a measure and $f$ a $\mu$-measurable function, we write $\mu(f)=\int f~d\mu$.
We write $\d=\frac 1 2(\d/\d x-i \d/\d y)$ and $\bar \d=\frac 1 2(\d/\d x+i \d/\d y)$ for the complex derivatives.

\section{Random normal matrix ensembles} \label{Sec1}

\subsection{The distribution of eigenvalues} Let $Q:\C\to \R\cup\{+\infty\}$ be a suitable lower semi-continuous function subject to the
growth condition
\begin{equation}\label{E11}\liminf_{z\to\infty} \frac{Q(z)}{\log|z|}>1.\end{equation}
We refer to $Q$ as the \textit{weight function} or the \textit{potential}.

Let $\mathcal{N}_n$ be the set of all $n\times n$ normal matrices $M$, i.e., $MM^*=M^*M$. The \textit{partition function}
on $\mathcal{N}_n$ associated with $Q$ is the function
\begin{equation*}\mathcal{Z}_n=\int_{\mathcal{N}_n}e^{-2n\trace Q(M)}~dM_n,\end{equation*}
where $dM_n$ is the Riemannian volume form on $\normalsp_n$ inherited from the space $\C^{n^2}$ of all $n\times n$ matrices, and where $\trace Q:\normalsp_n\to \R\cup\{+\infty\}$ is the random variable
$$\trace Q(M)=\sum_{\lambda_j~\in~\text{spec}(M)} Q(\lambda_j),$$
i.e., the usual trace of the matrix $Q(M)$. We equip $\mathcal{N}_n$ with the probability measure
\begin{equation*}d\mathcal{P}_n=\frac 1 {\mathcal{Z}_n}e^{-2n\trace Q(M)}~dM_n,\end{equation*}
and speak of the \textit{random normal matrix ensemble} or "RNM-ensemble'' associated with $Q$.

The measure $\mathcal{P}_n$ induces a measure $\Prob_n$ on the space $\C^n$ of eigenvalues,
 which is known as the \textit{density of states} in the external field $Q$; it is given by
\begin{equation*}d\Prob_n(\lambda)=\frac 1 {Z_n}
e^{-H_n(\lambda)}~dA_n(\lambda),\quad \lambda=(\lambda_j)_1^n\in \C^n.\end{equation*}
Here we have put
\begin{equation*}
H_n(\lambda)=\sum_{j\ne k}\log \frac 1 {\babs{\lambda_j-\lambda_k}}+2n\sum_{j=1}^n Q(\lambda_j),\end{equation*}
and $dA_n(\lambda)=d^2\lambda_1\cdots d^2\lambda_n$ denotes Lebesgue measure in $\C^n$, while
$Z_n$ is the normalizing constant giving $\Prob_n$ unit mass. By a slight abuse of language, we will refer to
$Z_n$ as the partition function of the ensemble.

Notice that $H_n$ is the energy (Hamiltonian) of a system of $n$ identical point charges in the plane located at the points $\lambda_j$, under influence
of the external field $2nQ$. In this interpretation, $\mathbf{P}_n$ is the law of the Coulomb gas in the external magnetic field $2nQ$ (at inverse temperature $\beta=2$). In particular, this explains the repelling nature of the eigenvalues of random normal matrices; they tend to be very spread out
in the vicinity of the droplet, just like point charges would.

Consider the $n$-point configuration ("set'' with possible repeated elements) $\{\lambda_j\}_1^n$ of eigenvalues of a normal matrix picked randomly with respect to $\mathcal{P}_n$. In an obvious manner, the measure $\Prob_n$
induces a probability law on the $n$-point configuration space; this is the law of the \textit{$n$-point process} $\Psi_n=\{\lambda_j\}_1^n$ associated to $Q$.

It is well-known that the process $\Psi_n$ is \textit{determinantal}. This means that there exists a Hermitian function
$\CK_n$, called the \textit{correlation kernel of the process} such that the density of states can be represented in the form
\begin{equation*}d\Prob_n(\lambda)=\frac 1 {n!}\det\lpar \CK_n(\lambda_j,\lambda_k)\rpar_{j,k=1}^n dA_n(\lambda),\qquad \lambda\in\C^n.\end{equation*}
One has
\begin{equation*}\CK_n(z,w)=K_n(z,w)e^{-n\lpar Q(z)+Q(w)\rpar},\end{equation*}
where $K_n$ is the reproducing kernel of the space $\PP_n\lpar e^{-2nQ}\rpar$ of analytic
polynomials of degree at most $n-1$ with norm induced from the usual $L^2$ space on $\C$ associated with the weight function
$e^{-2nQ}$. Alternatively, we can regard $\mathbf{K}_n$ as the reproducing kernel for the subspace
\begin{equation*}W_n=\{pe^{-nQ};~p\text{ is an analytic polynomial of degree less than }n\}\subset L^2(\C).\end{equation*}
We have the frequently useful identities
$$f(z)=\int_\C f(w)\overline{\CK_n(z,w)}~d^2w,\quad f\in W_n,$$
and
$$\int_\C\CK_n(z,z)~d^2 z=n.$$

We refer to \cite{B}, \cite{S}, \cite{EF}, \cite{HM}, \cite{AHM2}, \cite{M} for more
details on point-processes and random matrices.

\subsection{The equilibrium measure and the droplet} We are interested in the asymptotic distribution of eigenvalues
as $n$, the size of the matrices, increases indefinitely. Let $u_n$ denote the one-point function of $\Prob_n$, i.e.,
\begin{equation*}u_n(\lambda)=\frac 1 n \mathbf{K}_n(\lambda,\lambda),\quad \lambda\in \C.\end{equation*}
With a suitable function $f$ on $\C$, we associate the
random variable $\Tr_n[f]$ on the probability space $(\C^n,\Prob_n)$ via
$$\Tr_n[f](\lambda)=\sum_{i=1}^n f(\lambda_i).$$
The expectation is given by
$$\E_n\lpar \Tr_n[f]\rpar=n\int_\C f\cdot u_n.$$
According to Johansson (see \cite{HM}) we have weak-star convergence of the measures
$$d\sigma_n(z)=u_n(z)d^2 z$$
to some probability measure $\sigma=\sigma(Q)$ on $\C$.

In fact, $\sigma$ is the Frostman equilibrium measure of the logarithmic potential theory with external field
$Q$. We briefly recall the definition and some basic properties of this probability measure,
cf. \cite{ST} and \cite{HM} for proofs and further details.

Let $S=\supp\sigma$ and assume that $Q$ is $\mathcal{C}^2$-smooth in some neighbourhood of $S$. Then $S$ is compact, $Q$ is subharmonic on $S$, and $\sigma$ is
absolutely continuous with density
\begin{equation*}u=\frac 1 {2\pi}\Delta Q\cdot \1_S.\end{equation*}
We refer to the compact set $S=S_Q$ as the \textit{droplet} corresponding to the external field $Q$.

Our present goal is to describe the fluctuations of the density field $\mu_n=\sum_{j=1}^n\delta_{\lambda_j}$ around the equilibrium. More precisely, we will study the distribution (linear statistic)
\begin{equation*}f\mapsto \mu_n(f)-n\sigma(f)=\Tr_n[f]-n\sigma(f),\qquad f\in\coity(\C).\end{equation*}
We will denote by $\nu_n$ the measure with density $n(u_n-u)$, i.e.,
\begin{equation*}\nu_n[f]=\mathbf{E}_n\left[\Tr_n[f]\right]-n\sigma(f)=n(\sigma_n-\sigma)(f),\qquad f\in\coity(\C).\end{equation*}

\subsection{Assumptions on the potential} To state the main results of the paper we make the following three assumptions:

\bs\no (A1) (smoothness) $Q$ is real analytic (written
$Q\in C^\omega$)  in some  neighborhood of the droplet $S=S_Q$;

\bs\no (A2) (regularity)  $\Delta Q\ne 0$ in $S$;

\bs\no (A3) (topology) $\pa S$ is a $C^\omega$-smooth Jordan curve.

\bs\no We will comment on the nature and consequences of these assumptions later.
Let us denote
$$ L=\log\Delta Q.$$
This function is well-defined and  $C^\omega$ in a   neighborhood of the droplet.

\subsection{The Neumann jump operator} We will use the following general system of notation. If $g$ is a continuous function defined in  a neighborhood of $S$, then we write  $g^S$ for the function on the Riemann sphere $\hat \C$ such that
 $g^S$ equals $g$ in $S$ while $g^S$ equals the
harmonic extension of $g\bigm|_{\d S}$ to $\hat \C\setminus S$ on that set.

If $g$ is smooth on $S$, then
$$\NN_\Omega g:=-\frac {\pa g|_S}{\pa n},\qquad \Omega:=\Int(S),$$
where $n$ is the (exterior) unit normal of $\Omega$.  We define  the normal derivative $\NN_{\Omega_*}g$ for the complementary domain
$\Omega_*:=\hat \C\sm S$ similarly. If both normal  derivatives exist,  then we define (Neumann's jump)
$$\NN g\equiv \NN_{\pa S}:=\NN_\Omega g+\NN_{\Omega_*} g.$$

\ms\no
By Green's formula we have the identity (of measures)
\begin{equation}\label{E1.1}\Delta g^S=\Delta g\cdot 1_\Omega+\NN (g^S)~ds,\end{equation}
where $ds$ is the arclength measure on $\pa S$.

\ms
We now verify \eqref{E1.1}. Let $\phi$ be a test function.  The left hand side in \eqref{E1.1}
applied to $\phi$ is
$$\int_\C\phi \Delta g^S=\int_\C g^S\Delta\phi  =\int_Sg\Delta\phi  +\int_{\C\sm S} g^S\Delta\phi ,$$
and the right hand side is
$$\int_S\phi\Delta g+\int\phi \NN(g^S)ds.$$
Thus we need to check that
$$\int_S(g\Delta\phi-\phi\Delta g)+\int_{\C\sm S}(g^S\Delta\phi-\phi\Delta g^S)=  \int\phi \NN(g^S)ds.$$
But the expression in the left hand side is
$$\int(g\pa_n\phi-\phi\pa_{n} g)ds+\int(g\pa_{n*}\phi-\phi\pa_{n*} g^S)ds=-\int(\phi\pa_{n} g+\phi\pa_{n*} g^S)ds=\int\phi~\NN(g^S)ds,$$
and \eqref{E1.1} is proved.

\subsection{Main results} We have the following results.

\begin{thm} \label{MT1} For all test functions $f\in C^\infty_0(\C)$, the limit $$\nu(f):=\lim_{n\to\infty}  \nu_n(f)$$
exists, and
$$ \nu(f)=\frac1{8\pi}\left[\int_S\Delta f+\int_S f\Delta L+\int_{\pa S} f~\NN( L^S)~ds\right].$$
\end{thm}

\ms
\no Equivalently,  we have
$$\nu_n\to\nu=\frac1{8\pi}\Delta\left(1_S +L^S\right)$$
in the sense of distributions.

\ms

\begin{thm} \label{MT2} Let $h\in C^\infty_0(\C)$ be a real-valued test function. Then, as $n\to\infty$,
$$\trace_n h-\E_n\trace_n h\;\to\; N\left(0,\frac1{2\pi}\int_{\C}\left|\nabla h^S \right|^2\right).$$
\end{thm}

\ms\no The last statement means convergence in distribution of random variables to a normal law with indicated expectation and variance. As noted in \cite{AHM2}, Section 7, the result can be restated in terms of convergence of random fields to a Gaussian
field on $S$ with free boundary conditions.

\subsection{Derivation of Theorem \ref{MT2}} We now show, using the variational approach due to Johansson \cite{J}, that the Gaussian convergence stated in Theorem \ref{MT2} follows from a generalized version of Theorem \ref{MT1}, which we now state.

Fix  a real-valued test function  $h$ and  consider the perturbed potentials
$$ \tilde Q_n:=Q-\frac1n h.$$
We denote by $\tilde u_n$  the one-point function of the
density of states $\tilde \Prob_n$ associated with the potential  $\tilde Q_n$. We write $\ti\sigma_n$ for the measure
with density $\ti u_n$ and
$\tilde\nu_n$ for the measure $n(\tilde \sigma_n-\sigma)$, i.e.,
\begin{equation}\label{E1.35}\tilde{\nu}_n[f]=n\ti\sigma_n(f)-n\sigma(f)=\tilde{\E}_n~\Tr_n[f]-n\sigma(f).\end{equation}

\ms\begin{thm} \label{MT3} For all $f\in C^\infty_0(\C)$ we have
$$\tilde \nu_n(f)-\nu_n(f) \to \frac1{2\pi}\int_\C\nabla f^S\cdot\nabla h^S.$$
\end{thm}

A proof of Theorem \ref{MT3} is given in Section \ref{Sec4}.

\newpage

\begin{claim} \label{ML1} Theorem \ref{MT2} is a consequence of Theorem \ref{MT3}.
\end{claim}

\begin{proof} Denote $X_n=\Tr_n h-\E_n\Tr_n h$ and write $a_n=\tilde \E_n X_n$. By Theorem \ref{MT3},
$$a_n\to a\quad \text{where}\quad a= \frac1{2\pi}\int_\C\nabla h^S\cdot\nabla h^S.$$
More generally, let $\lambda\ge 0$ be a parameter, and let $\tilde\E_{n,\lambda}$ denote expectation corresponding to the potential $Q-(\lambda h)/n$.
Write
$$F_n(\lambda):=\log\E_n ~e^{\lambda X_n},\qquad 0\le \lambda\le 1.$$
Since $F_n^\prime(\lambda)=\tilde\E_{n,\lambda} X_n$,
Theorem \ref{MT3} implies
\begin{equation}\label{E1.2}F'_n(\lambda)=\tilde\E_{n,\lambda} X_n\to \lambda a,\end{equation}
and
\begin{equation}\label{E1.3}\log\E_n e^{X_n}=F_n(1)=\int_0^1F'_n(\lambda)~d\lambda~\to~ \frac a2,\qquad \text{as}\quad n\to\infty.
\end{equation}
Here we use the convexity of the functions $F_n$,
$$F''_n(\lambda)=\tilde\E_{n,\lambda} X_n^2-\left(\tilde\E_{n,\lambda} X_n\right)^2\ge0,$$
which implies that the  convergence in \eqref{E1.2} is dominated:
$$0=F_n'(0)\le  F'_n(\lambda)\le F'_n(1).$$
Replacing $h$ by $t h$ where $t\in\R$, we get $\mathbf{E}_n(e^{t X_n})\to e^{t^2a/2}$ as $n\to\infty$, i.e.,
we have convergence of all moments of $X_n$ to the moments of the normal $N(0,a)$ distribution. It is well known
that this implies convergence in distribution, viz. Theorem \ref{MT2} follows.
\end{proof}

\subsection{Comments}

\subsubsection*{(a) Related Work} The one-dimensional analog
of the weighted RNM theory is the more well-known random Hermitian matrix theory, which was studied by Johansson in the important paper
\cite{J}. Indeed, Johansson obtains results not only random Hermitian matrix ensembles, but for more general (one-dimensional) $\beta$-ensembles. The paper \cite{J} was one of our main sources of inspiration for the present work.

In \cite{AHM2}, it was shown that the convergence in theorems \ref{MT1} and \ref{MT2} holds for test functions supported in the interior of the droplet. See also \cite{B2}.
In \cite{AHM2}, we also announced theorems \ref{MT1} and \ref{MT2} and proved several consequences of them, e.g. the convergence of Berezin measures, rooted at a point in the exterior of $S$, to a harmonic measure.

Rider and Virág \cite{RV} proved theorems \ref{MT1} and \ref{MT2} in the special case $Q(z)=\babs{z}^2$ (the \textit{Ginibre ensemble}). The paper \cite{BS} contains results in this direction for $\beta$-Ginibre ensembles for some special values of $\beta$.

Our main technique, the method of Ward identities, is common practice in field theories. In this method, one uses reparametrization invariance
of the partition function to deduce exact relations satisfied by the joint intensity functions of the ensemble.
In particular, the method was applied on the physical level by Wiegmann, Zabrodin et al. to study RNM ensembles as well as more
general OCP ensembles. See e.g. the papers \cite{WZ1}, \cite{WZ2}, \cite{WZ3}, \cite{Z}. A one-dimensional version of Ward's identity
was also used by Johansson in \cite{J}.

Finally, we wish to mention that one of the topics in this paper, the behaviour of fluctuations near the boundary, is analyzed from another perspective in the forthcoming paper \cite{AKM}.

\subsubsection*{(b) Assumptions on the potential} We here comment on the assumptions (A1)--(A3) which we require of the potential $Q$.

The $C^\omega$ assumption (A1) is natural for the study of fluctuation properties near the boundary of the droplet. (For test functions supported in the interior, one can do with less regularity.)

Using Sakai's theory \cite{S}, it can be shown that conditions (A1) and (A2) imply that $\d S$ is a union of  finitely many
$C^\omega$ curves with a finite number of singularities of known types. It is not difficult to
complete a proof using arguments
from \cite{HS}, Section 4.

We rule out singularities by the regularity assumption in (A3). What happens in the presence of singularities is probably an interesting topic, which we have not approached.

Without singularities the boundary of the droplet is a union of finitely many $C^\omega$ Jordan curves. Assumption (A3) means that we only consider the case of a single boundary component. Our methods extend without difficulty to the case of a multiply connected droplet. The disconnected case requires further analysis, and is not considered in this paper.

\subsubsection*{(c) Droplets and potential theory}
We here state the properties of droplets that will be needed for our analysis. Proofs for these properties can be found in
\cite{ST} and \cite{HM}.

We will write
$\check Q$ for the maximal subharmonic function $\le Q$ which grows as $\log|z|+O(1)$ when $\babs{z}\to\infty$.
We have that
$\check{Q}=Q$ on $S$ while $\check Q$ is $C^{1,1}$-smooth on $\C$ and
$$\check Q(z)=Q^S(z)+G(z,\infty),\qquad z\in \C\setminus S,$$
where $G$ is the classical Green's function of $\C\setminus S$. In particular, if
$$U^\sigma(z)=\int\log\frac1{|z-\zeta|}{d\sigma(\zeta)}$$
denotes the logarithmic potential of the equilibrium measure, then
\begin{equation}\label{qhe}\check Q+U^\sigma\equiv\const\end{equation}

The following proposition sums up some basic properties of the droplet and the function $\check{Q}$.

\begin{prop} \label{Prop1} Suppose $Q$ satisfies (A1)--(A3). Then $\pa S$ is a $C^\omega$ Jordan curve, $\check Q\in W^{2,\infty}(\C)$, and
therefore $$\pa \check Q=(\pa Q)^S.$$
Furthermore, we have
\begin{equation}\label{E1.4}Q(z)-\check Q(z)\asymp \delta(z)^2,\qquad z\not\in S, \quad\delta(z)\to0,\end{equation}
where $\delta(z)$ denotes the distance from $z$ to the droplet.\end{prop}

\subsubsection*{(d) Joint intensities} We will occasionally use the intensity $k$-point function of the process $\Psi_n$. This is
the function defined by
$$R_n^{(k)}(z_1,\ldots,z_k)
=\lim_{\eps\to 0}\frac {\Prob_n\lpar \bigcap_{j=1}^k \left\{\Psi_n\cap D(z_j,\eps)\ne \emptyset\right\}\rpar}
{\pi^k\eps^{2k}}=\det\lpar \mathbf{K}_n(z_i,z_j)\rpar_{i,j=1}^k.$$
In particular, $R_n^{(1)}=nu_n$.

\subsubsection*{(e) Organization of the paper} We will derive the following statement which combines theorems \ref{MT1} and \ref{MT3} (whence, by Lemma \ref{ML1}, it implies Theorem \ref{MT2}).

\bs\no{\it Main formula:} Let $\ti\nu_n$ be the measure defined in \eqref{E1.35}. Then
\begin{equation}\label{mainformula}\lim_{n\to\infty}\tilde \nu_n(f)=\frac1{8\pi}\left[\int_S\Delta f+\int_S f\Delta L+\int_{\pa S} f\NN(L^S)~ds\right]+\frac1{2\pi}\int_\C\nabla f^S\cdot\nabla h^S .\end{equation}

\ms\no Our proof of this formula is based on the  limit form of  Ward's identities which we discuss in the next section. To justify this limit form  we need to estimate certain error terms; this is done in Section \ref{Sec3}. In the proof, we refer to some basic estimates of polynomial Bergman kernels, which we collect the appendix. The proof of the main theorem is completed in Section \ref{Sec4}.

\section{Ward identities} \label{Sec2}

\subsection{Exact identities} For a suitable function $v$ on $\C$ we define a
random variable
$W_n^+[v]$
on the probability space $(\C^n,\Prob_n)$ by
$$W_n^+[v]=\frac12\sum_{j\ne k}\frac{v(\lambda_j)-v(\lambda_k)}{\lambda_j-\lambda_k}-2n~\Tr_n[v \pa Q]+
\Tr_n[\pa v].$$

\begin{prop} \label{Prop2.1} Let $v:\C\to\C$ be Lipschitz continuous with compact support.
Then
$$\E_n W_n^+[v]=0.$$
\end{prop}

\begin{proof}
We write
\begin{equation*}W_n^+[v]= I_n[v]-II_n[v]+III_n[v]\end{equation*}
where (almost everywhere)
\begin{equation*}I_n[v]\left( z\right)=\frac 1 2\sum_{j\ne k}^n\frac
{v\left( z_j\right)-v\left( z_k\right)} {z_j-z_k}
\quad ;\quad II_n[v]\left( z\right)=2\sum_{j=1}^n \d Q\left( z_j\right)~v\left( z_j\right)\quad ;\quad
 III_n[v](z)=\sum_{j=1}^n \d v\left( z_j\right).\end{equation*}
Let $\eps$ be a real parameter and put $z_j=\phi\lpar\zeta_j\rpar=\zeta_j+\eps~v\lpar \zeta_j\rpar/2$, $1\le j\le n$. Then, for $\eps>0$ small enough,
\begin{equation*}d^2 z_j=\lpar~ \babs{~\d \phi\lpar \zeta_j\rpar~}^{~2}-\babs{~\bar\d \phi\lpar\zeta_j\rpar~}^{~2}~\rpar~d^2\zeta_j=
\left[~1+\eps\re\d v\lpar\zeta_j\rpar+O\lpar \eps^2\rpar~\right]~d^2\zeta_j,
\end{equation*}
so that (with $III_n=III_n[v]$)
\begin{equation*}d A_n(z)=\left[~1+\eps~\re ~III_n\lpar \zeta\rpar+
O\lpar\eps^2\rpar~\right]~dA_n(\zeta).
\end{equation*}
Moreover,
\begin{equation*}\begin{split}\log\babs{~z_i-z_j~}^{~2}&=\log\babs{~\zeta_i-\zeta_j~}^{~2}+
\log\babs{~1+\frac \eps 2 \frac {v(\zeta_i)-v(\zeta_j)}
{\zeta_i-\zeta_j}~}^{~2}=\\
&=\log\babs{~\zeta_i-\zeta_j~}^{~2}+
\eps\re\frac {v(\zeta_i)-v(\zeta_j)}
{\zeta_i-\zeta_j}+O\lpar\eps^2\rpar,\\
\end{split}
\end{equation*}
so that
\begin{equation}\label{fs1}\sum_{j\ne k}^n\log\babs{~z_j-z_k~}^{~-1}=
\sum_{j\ne k}^n\log\babs{~\zeta_j-\zeta_k~}^{~-1}-
\eps  \re~ I_n(\zeta)+O\lpar\eps^2\rpar,\quad \text{as}\quad \eps\to 0.
\end{equation}
Finally,
\begin{equation*}Q\lpar z_j\rpar=Q\lpar \zeta_j+\frac \eps 2 v\lpar\zeta_j\rpar\rpar
=Q\lpar\zeta_j\rpar+\eps~ \re~\lpar \d Q\lpar \zeta_j\rpar~v\lpar\zeta_j\rpar\rpar,
\end{equation*}
so
\begin{equation}2n\label{fs2}\sum_{j=1}^n Q\lpar z_j\rpar=
2n\sum_{j=1}^n Q\lpar\zeta_j\rpar+\eps~ \re ~II_n(\zeta)+O\lpar\eps^2\rpar.
\end{equation}
Now \eqref{fs1} and \eqref{fs2} imply that the Hamiltonian
$H_n(z)=\sum_{j\ne k}\log\babs{~z_j-z_k~}^{~-1}+2n\sum_{j=1}^n Q(z_j)$ satisfies
\begin{equation}\label{fs3}H_n (z)=H_n(\zeta)
+\eps\cdot  \re~ \lpar- I_n(\zeta)+II_n(\zeta)\rpar+
O\lpar\eps^2\rpar.\end{equation}
It follows that
\begin{equation*}\begin{split}Z_n:&=\int_{\C^n}\e^{-H_n(z)}~
dA_n(z)=\int_{\C^n}\e^{-H_n(\zeta)-\eps~\re~\lpar- I_n(\zeta)+II_n(\zeta)\rpar+
O\lpar\eps^2\rpar}~\left[1+\eps \re III_n(\zeta)+O\lpar\eps^2\rpar\right]~
dA_n(\zeta).\\
\end{split}\end{equation*}
Since the integral is independent of $\eps$, the coefficient of $\eps$
in the right hand side must vanish, which means that
\begin{equation}\re \int_{\C^n}
\lpar III_n(\zeta)+I_n(\zeta)-
II_n(\zeta)\rpar
~\e^{-H_n(\zeta)}~dA_n(\zeta)=0,\end{equation}
or $\re \mathbf{E}_n ~W_n^+[v]=0$. Replacing $v$ by
$\imag v$ in the preceding argument gives $\im \mathbf{E}_n
~W_n^+[v]=0$ and the proposition follows.
 \end{proof}

\bs\no  Applying Proposition \ref{Prop2.1} to the potential $\tilde Q_n=Q-h/n$, we get the identity
\begin{equation}\label{E2.1}\tilde \E_n \tilde W_n^+[v]=0,\end{equation}
where
\begin{equation}\label{E2.2}\tilde W_n^+[v]=W_n^+[v]+2\Tr_n[v\pa h].\end{equation}
If we denote
$$B_n[v]=\frac1{2n}\sum_{i\ne j}\frac{v(\lambda_i)-v(\lambda_j)}{\lambda_i-\lambda_j}, $$
we can rewrite \eqref{E2.1} and \eqref{E2.2} as follows,
\begin{equation}\label{E2.3}\ti\E_n B_n[v]=2\ti\E_n ~\Tr_n[v\pa Q]-\ti\sigma_n(\pa v+2v\pa h),\end{equation}
where we recall that $\ti\sigma_n$ is the measure with density $\ti u_n$.

\subsection{Cauchy kernels}
For each $z\in\C$ let $k_z$ denote the function
$$k_z(\lambda)=\frac1{z-\lambda},$$
so $z\mapsto \sigma(k_z)$ is the Cauchy transform of the the measure $\sigma$. We have (see \eqref{qhe})
$$\sigma(k_z)=2\pa\check Q(z).$$
We will also consider the Cauchy integrals $\sigma_n(k_z)$ and $\tilde\sigma_n(k_z)$. We have
 $$\bp_z [\sigma_n(k_z)]=\pi u_n(z),\qquad \bp_z [\tilde \sigma_n(k_z)] =\pi \tilde u_n(z),\quad z\in\C,$$
and
$$\tilde\sigma_n(k_z)\to \sigma(k_z)$$
with uniform convergence on $\C$
 (the uniform convergence follows easily from the one-point function estimates in Lemma \ref{LemExt} and Theorem
 \ref{ThmK}).

 \bs
\no Let us now introduce the functions
$$D_n(z)=\nu_n(k_z)\qquad;\qquad \tilde D_n(z)=\tilde \nu_n(k_z).    $$
We have
\begin{equation}\label{E2.8}\tilde D_n(z)=n[\tilde \sigma_n(k_z)-2\pa\check Q(z)]\quad ,\quad
\bar\d\ti D_n=n\pi(\ti u_n-u),\end{equation}
 and if $f$ is a test function, then
\begin{equation}\label{tni}\tilde \nu_n(f)~=~\frac1\pi\int f\bp \tilde D_n~=~-\frac1\pi\int \bp f\cdot\tilde  D_n.\end{equation}
Let $\ti\K_n$ denote the correlation kernel with respect to $\ti Q_n$.
Using $\tilde D_n$,  we can rewrite the $B_n[v]$ term in the Ward identity as follows.

\ss

\begin{lem} \label{Lem2.2} One has that
$$\ti\E_n B_n[v]=2\int v \cdot \pa\check Q\cdot \ti\K_n+\int v  \ti D_n\ti u_n-\frac1{2n}\iint\frac{v(z)-v(w)}{z-w}~|\ti\K_n(z,w)|^2.$$
\end{lem}

\no (In the first integral $\ti \K_n(z)$ means the 1-point intensity $\ti R_n^{(1)}(z)=\ti \K_n(z,z)$.)

\bs\no
\begin{proof} We have
 $$\ti\E_n B_n[v]=\frac1{2n}\iint_{\C^2} \frac{v(z)-v(w)}{z-w}\ti R_n^{(2)}(z,w),$$
where
$$ \ti R^{(2)}_n(z,w)=\ti\K_n(z)\ti\K_n(w)-
|\ti\K_n(z,w)|^2.$$
The integral involving $\ti\K_n(z)\ti\K_n(w)$ is
 $$\frac1{n}\iint_{\C^2} \frac{v(z)}{z-w}\ti \K_n(z)\ti \K_n(w)=\int_\C v(z)\cdot\ti \K_n(z)\cdot \ti \sigma_n(k_z),$$
and by \eqref{E2.8}
$\ti \sigma_n(k_z)=\frac1n \ti D_n+2\pa\check Q.$
\end{proof}

\subsection{Limit form of Ward's identity} The main formula \eqref{mainformula} will be derived from Theorem \ref{ThmW} below. In this theorem we make the following assumptions on the vector field $v$:

\ms\no (i) $v$ is bounded on $\C$;

\ms\no (ii)  $v$  is Lip-continuous in  $\C$;

\ms\no (iii) $v$ is uniformly $C^2$-smooth in $\C\sm \pa S$.

\ms\no(The last condition means that the restriction of $v$ to $S$ and the restriction   to $(\C\sm S)\cup \pa S$ are both  $C^2$-smooth.)

\begin{thm} \label{ThmW} If $v$ satisfies (i)-(iii), then as $n\to\infty$,
$$\frac2{\pi}\int_Sv\tilde D_n~\pa\bp Q+\frac2\pi\int_{\C\sm S}v(\pa\check Q-\pa Q)~\bp \tilde D_n~\to -\frac12\sigma(\pa v)-2\sigma(v\pa h).$$
\end{thm}

\ms\no
Before we come to the proof, we check
that it is possible to integrate by parts in the second integral in Theorem \ref{ThmW}. To control the boundary term we can use the next lemma.

\ss\begin{lem} For every fixed $n$ we have
$$\left|\tilde D_n(z)\right|\lesssim \frac1{|z|^2},\qquad (z\to\infty).$$\end{lem}

\begin{proof} We have
$$\left|\frac{\tilde D_n(z)}{n}\right|=\left|\int\frac{(\tilde u_n-u)d^2\lambda}{z-\lambda}\right|=
\left|\int\left[\frac1{z-\lambda}-\frac1{z}\right](\tilde u_n-u)d^2\lambda\right|$$
Since
$$\frac1{z-\lambda}-\frac1{z}=\frac1{z^2}~\frac\lambda{1-\lambda/z},$$
we need to show that the integrals
$$\int \frac{|\lambda|~|\tilde u_n-u|}{|1-\lambda/z|}d^2\lambda$$
are uniformly bounded. To prove this, we only need the estimate
$\tilde u_n(\lambda)\lesssim\frac1{|\lambda|^3}$, which holds (for sufficiently large $n$) by the
growth assumption \eqref{E1.1} and the simple estimate
$\tilde u_n(\lambda)\le C\exp(-2n(Q(\lambda)-\check{Q}(\lambda)))$, which is given below in Lemma \ref{LemExt}.
\end{proof}

Using that $\d Q=\d\check{Q}$ in the interior of $S$, we deduce the following corollary of Theorem \ref{ThmW}.

\begin{cor} \label{Cor2.5} {\rm ("Limit Ward identity'')} Suppose that $v$ satisfies conditions (i)-(iii).
Then as $n\to\infty$ we have the convergence
$$\frac2{\pi}\int_\C\left[v\pa\bp Q+\bp v(\pa Q-\pa \check Q)\right]\ti D_n\to-\frac12\sigma(\pa v)-2\sigma(v\pa h). $$
\end{cor}

\subsection{Error terms and the proof of Theorem \ref{ThmW}}  Theorem \ref{ThmW} follows if we combine the expressions
for $\ti\E_n B_n[v]$ in \eqref{E2.3} and Lemma
\ref{Lem2.2}
and  use the following approximations of the last two terms in Lemma \ref{Lem2.2}.
More precisely, if we introduce the \textit{first error term}  by
\begin{equation}\label{eps1}\frac1{n}\iint\frac{v(z)-v(w)}{z-w}~|\ti\K_n(z,w)|^2=\ti \sigma_n(\pa v)+\varepsilon_n^1[v],\end{equation}
and the \textit{second error term} by
\begin{equation}\label{eps2}\varepsilon_n^2[v]=\pi\int v \ti D_n(\ti u_n-u)=- \frac 1 2\int \bp v~\frac{\ti D_n^2}n.\end{equation}
Using \eqref{E2.3}, Lemma \ref{Lem2.2}, and that $\d\check Q=\d Q$ a.e. on $S$, one deduces that
\begin{equation}\label{limitward}\frac 2 \pi\int_Sv\ti D_n\d\bar\d Q+\frac 2 \pi\int_{\C\sm S}
v(\d\check Q-\d Q)\bar\d\ti D_n=-\frac 1 2 \sigma(\d v)-2\sigma(v\d h)+\frac 1 2 \eps_n^1[v]-\frac 1 \pi\eps_n^2[v]+o(1),\end{equation}
where $o(1)=(\sigma-\ti\sigma_n)(\d v/2+2v\d h)$ converges to zero as $n\to\infty$ by the one-point function estimates in Lemma \ref{LemExt} and Theorem \ref{ThmK}.

In the next section we will show that
for each $v$ satisfying conditions (i)-(iii), the error terms $\eps_n^j[v]$ tend to zero as $n\to \infty$, which will finish the proof
of Theorem \ref{ThmW}.

\section{Estimates of the error terms} \label{Sec3}

\subsection{Estimates of the kernel $\bf \ti \K_n$} We will use two different estimates for the correlation kernel, one to handle the interior and another for the exterior of the droplet.

\subsubsection*{(a) Exterior estimate.} Recall that $\ti \K_n(z,w)$ is the kernel of the $n$-point process associated with potential $\ti Q_n=Q-h/n$; as usual, we write $\ti \K_n(z)=\ti \K_n(z,z)$. We have the following global estimate, which is particularly useful in the exterior of the droplet.

\begin{lem} \label{LemExt} For all $z\in\C$ we have
$$\ti \K_n(z)\lesssim ne^{-2n(Q- \check Q)(z)},$$
where the constant is independent of $n$ and $z$.
\end{lem}

\bs\no
This estimate has been recorded (see e.g. \cite{AHM1}, Section 3) for the kernels  $\K_n$, i.e. in the  case $h=0$. Since obviously
$$\|p\|_{e^{-2n\ti Q_n}}\asymp\|p\|_{e^{-2nQ}},$$ we have $\ti \K_n(z)\asymp \K_n(z)$
with a constant independent of $z$. Indeed,
 $\K_n(z)$ is the supremum of $\babs{p(z)}^2e^{-2nQ(z)}$ where $p$ is an analytic polynomial of degree
less than $n$ such that $\|p\|_{e^{-2nQ}}\le 1$, and we have an analogous supremum characterization
of $\ti \K_n(z)$.
Hence the case $h\ne0$ does not require any special treatment.

\ms\no In the following we write
$$\delta(z)=\dist(z,\pa S)$$
and
$$\delta_n=\frac{\log^2n}{\sqrt n}.$$
By our assumption on the droplet (see Proposition \ref{Prop1})  we have
$$Q(z)-\check Q(z)\gtrsim \delta^2(z),\qquad z\not\in S,\quad \delta(z)\to0.$$
In view of the growth assumption \eqref{E11}, it follows that for any $N>0$ there exists $C_N$ such that
$\ti \K_n(z)\lesssim C_Nn^{-N}$ when $z$ is outside the $\delta_n$-neighborhood of $S$.

\subsubsection*{(b) Interior estimate.} Recall that we assume that $Q$ is real analytic in some neighbourhood of $S$.  This means that we can extend $Q$ to a complex analytic function of two variables in some neighbourhood in $\C^2$ of the anti-diagonal
$$\{(z,\bar z):~ z\in S\}\subset\C^2.$$
We will use the same letter $Q$ for this extension, so
$$Q(z)=Q(z,\bar z).$$ We  have
$$Q(z,w)=\overline {Q(\bar w,\bar z)}$$
and
$$\pa_1Q(z,\bar z)=\pa Q(z),\quad \pa_1\pa_2Q(z,\bar z)=\pa\bp Q(z),\quad
\pa_1^2Q(z,\bar z)=\pa^2 Q(z),\quad\rm etc.$$
 With the help of this  extension, one can show that the leading contribution to the kernel $\K_n$ is of the form
\begin{equation}\label{ksharp}\K_n^\#(z,w)~=\frac2\pi(\pa_1\pa_2Q)(z,\bar w)~ne^{n[2Q(z,\bar w)-Q(z)-Q(w)]}.\end{equation}
In particular, we have
$$\K_n^\#(w, w)~=\frac{n\Delta Q(w)}{2\pi},\qquad (w\in S).$$
We shall use the following estimate in the interior.

\begin{thm}\label{ThmK}  If $z\in S$, $\delta(z)>2\delta_n$, and if $|z-w|<\delta_n$, then
$$ \babs{\ti\K_n(z, w)}~=~\babs{\K_n^\#(z, w)} +O(1),$$
where the constant in $O(1)$ depend on $Q$ and $h$ but not on $n$.
\end{thm}

\bs\no Similar types of expansions are discussed e.g. in \cite{B1}, \cite{A}, \cite{AHM1}. As there is no convenient reference for this
particular result, and to make the paper selfcontained, we include a proof in the appendix.

\bs\no We now turn to the proof that  the error terms $\eps^1_n[v]$ and $\eps^2_n[v]$ are negligible. See \eqref{eps1} and \eqref{eps2}. Our proof uses only the  estimates of the kernels $\ti\K_n$ mentioned above. Since the form of these estimates is the same for all  perturbation functions $h$, we can without loss of generality set $h=0$, which will simplify our notation -- no need to put tildes on numerous letters.

\subsection{First error term} We start with the observation that if $w\in S$ and $\delta(w)>2\delta_n$ then at short distances the so called
Berezin kernel rooted at $w$
$$B^{\langle w\rangle}_n(z)=\frac{|\K_n(z,w)|^2}{\K_n(w,w)}$$
is close  to the heat kernel
$$H^{\langle w\rangle}_n(z)=\frac1\pi cne^{-cn|z-w|^2},\qquad c:=2\pa\bp Q(w).$$
Both kernels determine probability measures indexed by $w$.
Most of the heat kernel measure is concentrated in the disc $D(w,\delta_n)$,
$$\int_{\C\sm D(w,\delta_n)} H^{\langle w\rangle}_n(z)~dA(z)~\lesssim ~\frac1{n^{N}},$$
where $N$ denotes an arbitrary (large) positive number.

\ms\begin{lem} \label{Lem3.3}Suppose that $w\in S$, $\delta(w)>2\delta_n$ and $|z-w|<\delta_n$. Then
$$ |B^{\langle w\rangle}_n(z)-H^{\langle w\rangle}_n(z)|~\lesssim~ n\delta_n.$$
\end{lem}
\begin{proof}  By Theorem \ref{ThmK} we have
$$B^{\langle w\rangle}_n(z)=\frac{|\K_n^\#(z,w)|^2}{\K_n^\#(w,w)}+O(1).$$
Next, we fix $w$ and apply Taylor's formula to the function $z\mapsto \K^\#(z,w)$ at $z= w$.
Using the explicit formula \eqref{ksharp} for this function, and that
\begin{align*}&Q(z,\bar w)+\overline{Q(z,\bar w)}-Q(z,\bar z)-Q(w,\bar w)=
[Q(z,\bar w)-Q(w,\bar w)]+[Q(w,\bar z)-Q(z,\bar z)]=\\
&\pa Q(w)(z-w)+\frac12\pa^2Q(w)(z-w)^2+\pa Q(z)(w-z)+\frac12\pa^2Q(z)(z-w)^2+\dots=\\
&[\pa Q(w)-\pa Q(z)](z-w)+\pa^2Q(w)(z-w)^2+\dots=-\pa\bp Q(w)|z-w|^2+\dots
\end{align*}
\ss\no
we get
\begin{align*}\frac{|\K_n^\#(z,w)|^2}{\K_n^\#(w,w)}&=\frac1\pi [(c+O(|z-w|)] ~ne^{-cn|z-w|^2+O(n|z-w|^3)}\\
&=H^{\langle w\rangle}_n(z)+O(n|z-w|),
\end{align*}
and the assertion follows.
\end{proof}

\begin{cor} \label{Cor3.4} If  $w\in S$ and $\delta(w)>2\delta_n$, then
$$\int_{\C\sm D(w,\delta_n)} B^{\langle w\rangle}_n(z)~dA(z)~\lesssim ~n\delta_n^3=o(1).$$
\end{cor}
\begin{proof} We write $D_n=D(w,\delta_n)$ and notice that
$$\int_{\C\sm D_n} B_n^{\langle w\rangle}=1-\int_{D_n}B_n^{\langle w\rangle}=1-\int_{D_n} H_n^{\langle w\rangle}+\int_{D_n}(H_n^{\langle w\rangle}-B_n^{\langle w\rangle})=\int_{\C\sm D_n}H_n^{\langle w\rangle}+\int_{D_n}(H_n^{\langle w\rangle}-B_n^{\langle w\rangle}).$$
The statement now follows from Lemma \ref{Lem3.3}. \end{proof}

\begin{prop} \label{Prop3.5} If $v$ is uniformly Lipschitz continuous on $\C$, then
$\epsilon^1_n[v]\to 0$ as $n\to\infty$.\end{prop}

\begin{proof} We represent the error term as follows:
$$\epsilon_n^1[v]=\int_{(w)} u_n(w)F_n(w)\quad ;
\quad F_n(w)=\int_{(z)}\left[\frac{v(z)-v(w)}{z-w}-\pa v(w)\right]~B^{\langle w\rangle}_n(z).$$
By the  assumption that $v$ is globally Lipschitzian, we have that
$$\babs{\int_{\{\delta(w)<2\delta_n\}} u_n(w)F_n(w)}~\lesssim~\int_{\{\delta(w)<2\delta_n\}} u_n(w)~=~o(1).$$
If $\delta(w)>2\delta_n$, then
$$\left|F_n(w)\right|\lesssim \left|\int_{z\in D(w, \delta_n)}\left[\frac{v(z)-v(w)}{z-w}-\pa v(w)\right]~B^{\langle w\rangle}_n(z)\right|+\const\int_{z\not\in D(w, \delta_n)}B^{\langle w\rangle}_n(z),$$
where the last term is $o(1)$ by Corollary \ref{Cor3.4}. Meanwhile, the integral over $D(w, \delta_n)$ is bounded by
$$
\left|\int_{z\in D(w, \delta_n)}\left[\frac{v(z)-v(w)}{z-w}-\pa v(w)\right]~H^{\langle w\rangle}_n(z)\right|
+\const\int_{D(w, \delta_n)}|B^{\langle w\rangle}_n(z)-H^{\langle w\rangle}_n(z)|,$$
where we can neglect the second term (see Lemma \ref{Lem3.3}).
Finally,
$$\frac{v(z)-v(w)}{z-w}-\pa v(w)=\bp v(w)~\frac{\bar z-\bar w}{z-w}+o(1),$$
(this is where we use the assumption $v\in C^1(S)$), so the bound of the first term is $o(1)$ by the radial symmetry of the heat kernel.
\end{proof}

\subsection{Second error term} We shall prove the following proposition.

\begin{prop}\label{Prop3.6} If $v$ is uniformly Lipschitzian, then
$$\epsilon^2_n[v]:=   - \frac 1 2 \int \bp v~\frac{ D_n^2}n   \to 0,\qquad \text{as}\quad  n\to\infty.$$\end{prop}

\no The proof will involve certain  estimates of the function
$$D_n(z)=\int_{\C}\frac{\K_n(\zeta)-\K_n^\#(\zeta)}{z-\zeta}~d^2\zeta.$$
(Here $\K^\#(\zeta)=nu\cdot 1_S$.) It will be convenient to split the integral into two parts:
$$D_n(z)=C_n(z)+R_n(z):=\lpar \int_{B_n}+\int_{\C\sm B_n}\rpar \frac{\K_n(\zeta)-\K_n^\#(\zeta)}{z-\zeta}~d^2\zeta,$$
where
$$B_n=\{z:~\delta(z)<2\delta_n\}.$$
By Theorem \ref{ThmK} and Lemma \ref{LemExt} we have $|\K_n-\K^\#_n|\lesssim1$ in $\C\sm B_n$, and therefore
$$|R_n|\lesssim 1.$$
Hence we only need to estimate
$C_n$, the Cauchy transform of a real measure supported in $B_n$ -- a narrow "ring" around $\pa S$.
We start with a simple uniform bound.

\begin{lem} \label{Lem3.7} The following estimate holds,
$$\left\|D_n\right\|_{L^\infty}\lesssim \sqrt n\log^3 n.$$\end{lem}
\begin{proof} This follows from the  trivial bound $|\K_n-\K^\#_n|\lesssim n$ and the following  estimate of the integral
$$\int_{B_n}\frac{d^2 \zeta}{|z-\zeta|}.$$
Without loosing generality, we can assume that $z=0$ and replace $B_n$ by the rectangle $|x|<1$, $|y|<\delta_n$. We have
$$\int_{B_n}\frac{d^2\zeta}{|\zeta|}=\int_{-1}^1dx\int_{-\delta_n}^{\delta_n}\frac{dy}{\sqrt{ x^2+y^2}}=I+II,$$
where $I$ is the integral over $\{|x|<\delta_n\}$ and $II$ is the remaining term. Passing to polar coordinates we get
$$I\asymp \int_0^{\delta_n}\frac{rdr}{r}=\delta_n,$$
and
$$II\lesssim \delta_n\int_{\delta_n}^1\frac{dx}x\asymp \delta_n|\log\delta_n|.$$
\end{proof}

\bs\no Lemma \ref{Lem3.7} gives us the following estimate of the second error term,
\begin{equation}\label{E3.2} \babs{\varepsilon^2_n[v]}\lesssim \log^6n\|v\|_{\rm Lip},\end{equation}
which comes rather close but is still weaker than what we want.  Our
strategy will be to use \eqref{E3.2} and iterate the argument with Ward's identity. This will give a better estimate
in the \textit{interior} of the droplet.

\begin{lem} \label{Lem3.8} We have that
$$|C_n(z)|\lesssim \frac{\log^6n}{\delta(z)^3},\qquad z\in  S.$$\end{lem}

\begin{proof} Let $\psi$ be a function of Lipschitz norm less than 1 supported inside the droplet. Then
we have  $$\varepsilon^{1}_n[\psi]\lesssim 1, \qquad \varepsilon^{2}_n[\psi]\lesssim \log^6n, $$
where the constants don't depend on $\psi$. (The first estimate follows from Proposition \ref{Prop3.5}, and the second one is just \eqref{E3.2}). This means that the error $\eps_n:=\eps_n^1+\eps_n^2$
in the identity \eqref{limitward} is bounded by  $\log^6n$ for all such $\psi$, i.e. (since $\d Q=\d\check{Q}$ a.e. on $S$),
$$\left|\int \psi D_n \Delta Q \right|= O(1)+ \left|\varepsilon_n[\psi]\right|\lesssim \log^6n,
$$ and therefore,
$$\left|\int \psi C_n \Delta Q\right|\lesssim \log^6n$$
For $z\in S$ with $\delta(z)>4\delta_n$, we now set $2\delta=\delta(z)$ and consider the function
$$\psi(\zeta)=\max\left\{\frac{\delta-|\zeta-z|}{\Delta Q(\zeta)},~0\right\}.$$
Then $\psi$ has Lipschitz norm $\asymp 1$, and by analyticity of $C_n$ we have the mean value identity
$$\int \psi C_n \Delta Q=2\pi C_n(z)\int_0^\delta(\delta-r)rdr=\pi C_n(z)\delta^3/3.$$
We conclude that
$|C_n(z)|\lesssim \delta^{-3}\log^6n.$
\end{proof}

\bs\no Finally, we need an  estimate of $C_n$ in the \textit{exterior} of the droplet. This will be done in the next subsection  by reflecting  the previous interior estimate in the curve $\Gamma:=\pa S$.

Let us fix some sufficiently small positive number, e.g. $\varepsilon=\frac1{10}$ will do, and define
$$\gamma_n=n^{-\varepsilon}.$$
Denote
$$\Gamma_n=\{\zeta+\gamma_n\nu(\zeta):~\zeta\in\Gamma\},$$
where
 $\nu(\zeta)$ is the unit normal vector to $\Gamma$ at $\zeta\in\Gamma$ pointing outside from $S$.
We will write int$\Gamma_n$ and ext$\Gamma_n$ for the respective components of $\C\setminus\Gamma_n$. In
the following, the notation
$a\prec b$ will mean
inequality up to a multiplicative constant factor times some power of $\log n$ (thus e.g. $1\prec \log^2 n$).

Let $L^2(\Gamma_n)$ be the usual $L^2$ space of functions on $\Gamma_n$ with respect to arclength measure.
We will use the following lemma.

\begin{lem} \label{Lem3.9} We have that
$$\|C_n\|^2_{L^2(\Gamma_n)}\prec n\gamma_n.$$\end{lem}

\ms\no Given this estimate, we can complete the proof of Proposition \ref{Prop3.6} as follows.

\begin{proof}[Proof of Proposition \ref{Prop3.6}] Applying Green's formula to the expression for $\eps_n^2[v]$ (see
\eqref{eps2}),
using that $D_n$ is, to negligible terms, analytic in ext$\Gamma_n$, we find that
$$\babs{\eps_n^2[v]}\lesssim \frac1n\|D_n\|^2_{L^2({\rm int}\Gamma_n)}+\frac1n\|D_n\|^2_{L^2(\Gamma_n)}+o(1).$$
The second term is taken care of by Lemma \ref{Lem3.9}. To estimate the first term
denote $$A_n=\{\delta(z)<\gamma_n\}.$$
The area of $A_n$ is $\asymp \gamma_n$, and in $S\sm A_n$ we have
$|D_n(z)|\prec \gamma_n^{-3}$ (Lemma \ref{Lem3.8}). We now apply the uniform bound $|D_n|\prec \sqrt n$ in $A_n$ (Lemma \ref{Lem3.7}).
 It follows that
$$\|D_n\|^2_{L^2({\rm int}\Gamma_n)}=\int_{A_n}+\int_{S\sm A_n}\prec n|A_n|+\gamma_n^{-6},$$
whence
$$\|D_n\|^2_{L^2({\rm int}\Gamma_n)}=o(n).$$
This finishes the proof of the proposition.
\end{proof}

\subsection{Proof of Lemma \ref{Lem3.9}} Let us first establish  the following fact:
\begin{equation}\label{fact}\left|\im \left[\nu(\zeta)~C_n(\zeta+\gamma_n\nu(\zeta))\right]\right|~\prec ~\sqrt{n\gamma_n},\qquad \zeta\in\Gamma.\end{equation}

\ms
\begin{proof} Without loss of generality, assume that $\zeta=0$ and $\nu(\zeta)=i$.

The tangent to $\Gamma$ at $0$ is horizontal, so $\Gamma$ is the graph of $y=y(x)$ where $y(x)=O(x^2)$ as $x\to0$. We will show that
\begin{equation}\label{E3.4}\left|\re[C_n(i\gamma_n)-C_n(-i\gamma_n)]\right|~\prec~\sqrt{n\gamma_n} .\end{equation}
This implies the desired estimate \eqref{fact}, because by Lemma \ref{Lem3.8}
$$ |C_n(-i\gamma_n)|\prec \gamma_n^{-3}\le \sqrt{n\gamma_n}.$$
To prove \eqref{E3.4} we notice that
$$I:=\re[C_n(i\gamma_n)-C_n(-i\gamma_n)]=\int_{B_n}\re\left[\frac{1}{z-i\gamma_n}-\frac{1}{z+i\gamma_n}\right]\rho_n(z)dA(z),$$
where we have put $\rho_n=\K_n-\K_n^\#$, viz. $|\rho_n|\prec n$.

\ms\no We next subdivide the belt $B_n=\{\delta(z)<2\delta_n\}$ into two parts:
$$B_n'=B_n\cap\{|x|\le\sqrt \gamma_n\},\qquad B''_n=B_n\sm B'_n.$$
Clearly,
\begin{equation}\label{E3.5}|I|\lesssim n\int_{B'_n}\left|\frac{1}{z-i\gamma_n}-\frac{1}{\bar z-i\gamma_n}\right|+n \int_{B''_n}\left|\frac{1}{z-i\gamma_n}-\frac{1}{ z+i\gamma_n}\right|     .\end{equation}

\ss\no
The integral over $B'_n$ in the right hand side of \eqref{E3.5} is estimated by
$$\int_{B'_n}\asymp \int_{B'_n}\frac{|y|}{x^2+\gamma_n^2}\lesssim |B'_n|\asymp \delta_n\sqrt\gamma_n$$
because if $z=x+iy\in B_n'$ then $|y|\lesssim x^2+\delta_n\le x^2+\gamma_n^2.$

\ms\no We estimate the integral over $B''_n$ in \eqref{E3.5} by
$$\int_{B''_n}\asymp \int_{B''_n}\frac{\gamma_n}{x^2}\prec \delta_n\gamma_n\int_{\sqrt\gamma_n}^1\frac{dx}{x^2}\asymp \delta_n\sqrt\gamma_n.$$
It follows that
$$|I|\prec n\delta_n\sqrt\gamma_n\prec\sqrt {n\gamma_n}.$$
This establishes \eqref{E3.4}, and, as a consequence, \eqref{fact}.
\end{proof}

\bs\no To finish the proof of Lemma \ref{Lem3.9} we denote by $\nu_n(\cdot) $ the outer unit normal of $\Gamma_n$. Using \eqref{fact}  and Lemma \ref{Lem3.7} we deduce that
$$|\im~\nu_n C_n|\prec \sqrt {n\gamma_n} \qquad {\rm on}\quad\Gamma_n.$$
Next let $\D_*$ be the exterior of the closed unit disk and
consider the conformal map $$\phi_n: ~{\rm ext}(\Gamma_n)\to \D_*,\qquad \infty\mapsto\infty.$$
We put
$$F_n=\frac{\phi_nC_n}{\phi'_n}.$$
Then $F_n$ is analytic in ext$(\Gamma_n)$ including infinity, and we  have
\begin{equation}\label{E3.6}\|\im F_n\|^2_{L^2(\Gamma_n)}\prec \sqrt {n\gamma_n}.\end{equation}
To see this note that
$$\im F_n=\frac{\im[\nu_nC_n]}{|\phi'_n|},$$
and recall that we have assumed that $\Gamma$ is regular (A3), which means that $\babs{\phi'_n}$ is bounded below by a positive constant.

Now note that $\phi_n(z)/\phi_n^\prime(z)=r_nz+O(1)$ as $z\to\infty$, where the $r_n$ are uniformly bounded. This gives
$$F_n(\infty)=r_n \int_{B_n}(\K_n-\K_n^\#)=r_n\int_{\C\sm B_n}(\K^\#_n-\K_n)= O(1),$$
where we have used Lemma \ref{LemExt} and Theorem \ref{ThmK} to bound the integrand. Therefore, by \eqref{E3.6}, since the harmonic
conjugation operator is bounded on $L^2(\Gamma_n)$,
$$\|\re F_n\|^2_{L^2(\Gamma_n)}\asymp\|\im F_n\|^2_{L^2(\Gamma_n)}+O(1)  \prec \sqrt {n\gamma_n}.$$
This completes the proof of Lemma \ref{Lem3.9}.

\section{Proof of the main formula} \label{Sec4}

\bs\no  In this section we will use the limit form of Ward's identity (Corollary \ref{Cor2.5}) to derive our
main formula \eqref{mainformula}: for every test function $f$ the limit $\tilde\nu(f):=\lim_{n\to\infty}\tilde\nu_n(f)$ exists
and equals
\begin{equation}\label{mfo}\tilde \nu(f)=\frac1{8\pi}\left[\int_S\Delta f+\int_S f\Delta L+\int_{\pa S} f\NN(L^S) ds\right]+\frac1{2\pi}\int_\C\nabla f^S\cdot\nabla h^S .\end{equation}

\subsection{Decomposition of $\bf  f$} The following statement uses our assumption that $\pa S$ is a (real analytic) Jordan curve.

\bs\begin{lem} Let $f\in C^\infty_0(\C)$. Then $f$ has the following  representation:
$$f=f_++f_-+f_0,$$
where

\ss\no (i) all three functions are smooth on $\C$;

\ss\no (ii) $\bp f_+=0$ and $\pa f_-=0$ in $\C\sm S$;

\ss\no (iii) $f_\pm=O(1)$ at $\infty$;

\ss\no (iv) $f_0=0$ on $\pa S$.
\end{lem}

\begin{proof}  Consider the inverse conformal maps
$$\phi: \D_*\to \C\sm S,\qquad \psi:\C\sm S\to \D_*,\qquad \infty\mapsto\infty,$$ where $\D_*=\{|z|>1\}$. On the unit circle $\T$, we have
$$F:=f\circ \phi=\sum_{-\infty}^\infty a_n\zeta^n\in C^\infty(\T).$$
The functions  $$F_+(z)=\sum_{-\infty}^0 a_nz^n,\qquad F_-(z)=\sum_{1}^\infty\frac{ a_{n}}{\bar z^{n}}, \qquad (z\in \D_*)$$
are $C^\infty$ up to the boundary so we can
extend them to some smooth functions $F_\pm$ in $\C$. The conformal map  $\psi$ also extends to a smooth function  $\psi:\C\to\C$. It follows that
$$f_\pm:=F_\pm\circ\psi\in C^\infty_0(\C),  $$
and  $f_\pm$ satisfy (ii)--(iii). Finally, we  set
$$f_0=f-f_+-f_-.$$
\end{proof}

\bs\no{\bf Conclusion.} It is enough to prove the main formula \eqref{mfo} only for functions of the form $f=f_++f_-+f_0$ as in the last  lemma with an {\it additional}
assumption that $f_0$ is supported inside any given neighborhood of the droplet $S$.

\ms\no Indeed, either side of the formula \eqref{mfo} will not change if we "kill" $f_0$ outside the neighborhood. The justification is immediate by Lemma \ref{LemExt}.

\ms\no In what follows we will choose a neighborhood $O$ of $S$ such that the potential $Q$ is real analytic, strictly subharmonic in $O$, and
 $$\pa Q\ne\pa \check Q\quad{\rm in}\quad O\setminus S,$$ and will assume $\supp(f_0)\subset O$.

\bs\subsection{The choice of the vector field in Ward's identity} We will now compute the limit
$$\tilde \nu(f):=\lim\tilde \nu_n(f)$$
(and prove its existence) in the case where $$f=f_++f_0.$$ To apply the limit Ward identity
\begin{equation}\label{LWI}\frac2{\pi}\int_\C\left[v\pa\bp Q+\bp v(\pa Q-\pa \check Q)\right]\ti D_n\to-\frac12\sigma(\pa v)-2\sigma(v\pa h), \quad
(n\to\infty),\end{equation}
(see Corollary \ref{Cor2.5}), we set
$$v=v_++v_0,$$
where
$$v_0=\frac{\bp f_0}{\pa\bp Q}\cdot 1_S+
\frac{ f_0}{\pa Q-\pa \check Q}\cdot 1_ {\C\sm S},$$
and
$$v_+=\frac{\bp f_+}{\pa\bp Q}\cdot 1_S.$$
This gives
$$v=\frac{\bp f}{\pa\bp Q}\cdot 1_S+\frac{ f_0}{\pa Q-\pa \check Q}\cdot 1_ {\C\sm S}.$$
But in $\C\sm \pa S$ we have
$$v\pa\bp Q+\bp v\cdot \pa(Q-\check Q)=\bp f,$$
so comparing with \eqref{tni}, we find that
\begin{equation}\label{2tn}-2\tilde\nu_n(f)=\frac2{\pi}\int_\C\left[v\pa\bp Q+\bp v(\pa Q-\pa \check Q)\right]\ti D_n.\end{equation}
However, to justify that \eqref{LWI} holds, we must check that $v$ satisfies the conditions (i)-(iii) of Corollary \ref{Cor2.5}.

\ms
\begin{lem} The vector field $v$ defined above is Lip$(\C)$ and the restrictions of  $v$ to $S$ and to $S_*:= (\C\sm S)\cup\pa S$ are $C^\infty$.\end{lem}

\ms\begin{proof} We need to check the following items:

\ms\no (i) $v|_{S_*}$ is smooth, and (i$'$) $v|_{S}$ is smooth;

\ms\no (ii) $v_0$ is continuous on $\pa S$, and (ii$'$) same for $v_+$.

\ms\no The items (i$'$) and (ii$'$) are of course trivial. (E.g., $\bp f_+\cdot 1_S=\bp f_+$.)

\bs\no Proof of (i). We have $v=f_0/g$ in $\C\sm S$ where $g=\pa Q-\pa\check Q$. Since the statement is local, we consider
 a conformal map $\phi$ that takes a neighbourhood of a boundary point in $S$ onto a neighbourhood of a point in $\R$ and takes (parts of) $\pa S$ to $\R$. If we denote $F=f_0\circ \phi$ and $G=g\circ \phi$, then $F=0$ and $G=0$ on $\R$. Moreover, $G$ is real analytic
with non-vanishing derivative $G_y$. Thus it is enough to check that
$$H(x,y)=\frac{F(x,y)}y$$
has bounded derivatives of all orders. We will go through the details for $H, H_y, H_{yy}, \dots$. Applying the same argument to $H_x$, we get the boundedness of the derivatives $H_x, H_{xy}, H_{yy}, \dots$, etc.

\ms\no Let us show, e.g., that $H':=H_y$ is bounded. We have
$$H'=\frac{yF'-F}{y^2}=\frac{y(F'_0+O(y))-(yF'_0+O(y^2))}{y^2}=O(1),$$
where $F_0':=F'(\cdot, 0)$ and all big $O$'s are uniform in $x$. (They come from the bounds for the derivatives of $F$.) Similarly,
$$H''=\frac{y^2F''-2yF'+2F}{y^3}.$$
The numerator is
$$y^2(F''_0+O(y))-2y(F'_0+yF_0''+O(y^2))+2(yF'_0+\frac12y^2F_0''+O(y^3))=O(y^3),$$
etc. (We can actually stop here because we only need $C^2$ smoothness to apply Theorem \ref{ThmW}.)

\bs\no Proof of (ii).  Let   $n=n(\zeta)$  be the exterior unit normal with respect to $S$. We have
$$f_0(\zeta+\delta n)\sim \delta\pa_nf_0(\zeta)=2\delta\cdot (\bp f_0)(\zeta)\cdot \overline{n(\zeta)},\qquad
\text{as}\quad \delta\downarrow 0.$$ Similarly, if
$g:=\pa Q-\pa\check Q$, so $g=0$ on $\pa S$ and $\bp g=\pa\bp Q$ in $\C\sm S$, then
$$g(\zeta+\delta n)\sim \delta\pa_ng(\zeta)=2\delta\cdot (\bp g)(\zeta)\cdot\overline{n(\zeta)},\quad
\text{as}\quad \delta \downarrow 0,$$
where $\bp g(\zeta)$ denotes the $\bp$-derivative in the exterior sense. It follows that
$$\frac {f_0(\zeta+\delta n)}{g(\zeta+\delta n)}\sim\frac{\bp f_0(\zeta)}{\pa\bp Q(\zeta)},\qquad (\delta\downarrow 0),$$
which proves the continuity of $v_0$.
\end{proof}

We have established that $v=v_0+v_+$ satisfies conditions (i)-(iii) of Corollary \ref{Cor2.5}. Thus the convergence in \eqref{LWI} holds, and by \eqref{2tn} we conclude the following result.

\begin{cor} If $f=f_0+f_+$, then
$$\ti\nu(f)=\frac14\sigma(\pa v)+\sigma(v\pa h).$$\end{cor}

\bs

\newpage

\subsection{Conclusion of the proof}

\subsubsection*{(a)} Let us now consider the general case
$$f=f_++f_0+f_-.$$
By the last corollary we have
$$ \ti\nu(f_+)=\frac14\sigma(\pa v_+)+\sigma(v_+\pa h), \qquad v_+:=\frac{\bp f_+}{\pa\bp Q}\cdot 1_S .$$
Using complex conjugation we get a similar expression for $\ti\nu(f_-)$:
 $$\ti\nu(f_-)=\frac14\sigma(\bp v_-)+\sigma(v_-\bp h),\qquad
v_-:=\frac{\pa f_-}{\pa\bp Q}\cdot 1_S.$$
Indeed,
$$\ti\nu(f_-)=\overline{\ti\nu(\overline{f_-})}=\overline{\frac14\sigma(\pa \bar v_-)+\sigma(\bar v_-\pa h)}=\frac14\sigma(\bp v_-)+\sigma(v_-\bp h).$$
(Recall that $h$ is real-valued.)

\ms\no Summing up we get
\begin{equation}\label{sum1}\nu(f)= \frac14\left[\sigma(\pa v_0)+\sigma(\pa v_+)+\sigma(\bp v_-)\right]\end{equation}
and
\begin{equation}\label{sum2}\ti \nu(f)-\nu(f)=\sigma(v_+\pa h)+
\sigma(v_0\pa h)+\sigma(v_-\bp h).\end{equation}

\subsubsection*{ (b) Computation of  $\nu(f)$.} Recall that
$$d\sigma(z)=\frac 1 {2\pi}\Delta Q(z)\1_S(z)d^2 z\quad ,\quad L=\log\Delta Q.$$
Using \eqref{sum1} we compute
\begin{align*}\nu(f)&=\frac1{2\pi}\int_S\pa\left(\frac{\bp f_0+\bp f_+}{\pa\bp Q}\right)~\pa\bp Q+\frac1{2\pi}\int_S\bp\left(\frac{\pa f_-}{\pa\bp Q}\right)~\pa\bp Q\\
&=\frac1{2\pi}\int_S\pa\left(\frac{\bp f_0}{\pa\bp Q}\cdot \pa\bp Q\right)-\frac1{2\pi}\int_S  \frac{\bp f_0}{\pa\bp Q}~\pa(\pa\bp Q)+\frac1{2\pi}\int_S\pa\left(\frac{\bp f_+}{\pa\bp Q}\right)~\pa\bp Q+\frac1{2\pi}\int_S\bp\left(\frac{\pa f_-}{\pa\bp Q}\right)~\pa\bp Q\\
&=\frac1{2\pi}\int_S\pa\bp f
-\frac1{2\pi}\int_S\bp f_0~\pa L-\frac1{2\pi}\int_{S}\bp f_+~\pa  L-\frac1{2\pi}\int_{S}\pa f_-~\bp L
\end{align*}
At this point, let us modify  $L$ outside some neighborhood of $S$ to get a smooth function with compact support. We will still use the notation $L$ for the modified function. The last expression clearly does not change as a result of this modification. We can now transform the integrals involving $L$ as follows:
\begin{align*}
-\int_S \bp f_0~\pa L-&\int_{\C}\bp f_+~\pa  L-\int_{\C}\pa f_-~\bp L
=\int_S f_0~\pa\bp L+\int_{\C} (f_++f_-)~\pa\bp  L\\
&=\int_S f~\pa\bp L+\int_{\C\setminus S} f^S~\pa\bp  L,
\end{align*}
and we conclude that
$$\nu(f)=\frac1{8\pi}\left[\int_S\Delta f+\int_Sf\Delta L+\int_{\C\sm S}f^S\Delta L\right]$$

\bs\no {\bf Note.} The formula for $\nu(f)$ was stated in this form in \cite{AHM2}.

\bs\no Let us finally express  the last integral in terms of  Neumann's jump. We have
\begin{align*}\int_{\C\sm S}f^S\Delta  L&= \int_{\C\sm S} \left(f^S\Delta L-  L\Delta f^S \right)\\
&=  \int_{\pa S}  \left(f^S\cdot\pa_{n_*} L- \pa_{n_*} f^S\cdot L^S \right)~ds\\
&=  \int_{\pa S}  \left(f^S\cdot\pa_{n_*} L-  f^S\cdot\pa_{n_*} L^S \right)~ds\\
&= \int_{\pa S} f\NN(L^S)~ds
\end{align*}
In conclusion,
\begin{equation}\label{sum3}\nu(f)=\frac1{8\pi}\left[\int_S\Delta f+\int_Sf\Delta L+\int_{\pa S}f\NN(L^S)\right].\end{equation}

\subsubsection*{(c)  Computation of  $[\ti\nu(f)-\nu(f)]$} Using the identity \eqref{sum2}, we can deduce that
\begin{align*}\ti\nu(f)-\nu(f)&=\frac2\pi\left[\int_S\bp f_+\pa h+\int_S\pa f_-\bp h+\int_S\bp f_0\pa h
\right]\\
&=\frac1{2\pi}\int \nabla f^S\cdot\nabla h^S  .\end{align*}
This is because
$$\int_S\bp f_+\pa h=\int_\C\bp f_+\pa h=-\frac14\int_\C  f_+\Delta h=\frac14\int_\C  \nabla f_+\cdot \nabla h,$$
and similarly
$$\int_S\pa f_-\bp h=\frac14\int_\C  \nabla f_-\cdot \nabla h.$$
On the other hand,
$$\int_S\bp f_0\pa h=-\frac14\int_S  f_0 \Delta h=  \frac14\int_S  \nabla f_0\cdot \nabla h     .$$
Therefore,
$$\ti\nu(f)-\nu(f)=\frac1{2\pi}\left[\int_S  \nabla f\cdot \nabla h +\int_{\C\sm S}  \nabla f^S\cdot \nabla h\right],$$
and this is equal to
$$\frac1{2\pi}\int_{\C}  \nabla f^S\cdot \nabla h=\frac1{2\pi}\int_{\C}  \nabla f^S\cdot \nabla h^S.$$
Applying \eqref{sum3} we find that
$$\ti\nu(f)=\frac1{8\pi}\left[\int_S\Delta f+\int_Sf\Delta L+\int_{\pa S}f\NN(L^S)\right]+\frac1{2\pi}\int_{\C}  \nabla f^S\cdot \nabla h^S,$$
and the main formula \eqref{mfo} has been completely established. q.e.d.

\section*{Appendix: Bulk asymptotics for the correlation kernel}

\subsection*{
Polynomial Bergman spaces} For a suitable (extended) real valued function $\phi$, we denote by $L^2_\phi$ the space normed by
$\|f\|_\phi^2=\int_\C|f|^2e^{-2\phi}$. We denote by $A^2_\phi$ the subspace of $L^2_\phi$ consisting of a.e. entire functions;
$\PP_n(e^{-2\phi})$ denotes the subspace consisting of analytic polynomials of degree at most $n-1$.

Now consider a potential $Q$, real analytic and strictly subharmonic in some neighborhood of the droplet $S$, and subject to the
usual growth condition. We put
$$\ti Q\equiv \ti Q_n=Q-\frac1n h,$$
where $h$ is a smooth bounded real function.

We denote by $K$ the reproducing kernel for the space $\PP_n(e^{-2nQ})$, and write
$K_w(z)=K(z,w)$. The corresponding orthogonal projection is denoted by
$$P_n:L^2_{nQ}\to \PP_n(e^{-2nQ})\quad :\quad f\mapsto \left( f,K_w\right)_{nQ}.$$
The map $\ti P_n:L^2_{n\ti Q}\to  \PP_n(e^{-2n\ti Q})$ is defined similarly, using the reproducing kernel $\ti K$ for the space
$\PP_n(e^{-2n\ti Q})$.

\ms\no
We define approximate kernels and Bergman  projection as follows.
In the case $h=0$, the well-known first order approximation inside the droplet is given by the expression
$$K_w^\#(z)=\frac2\pi(\pa_1\pa_2Q)(z,\bar w)~n
e^{2nQ(z,\bar w)},$$
where $Q(\cdot,\cdot)$ is the complex  analytic function of two variables satisfying $$Q(w,\bar w)=Q(w).$$ If the perturbation $h\ne0$
is a real-analytic function, we can
 just replace $Q$ by $\ti Q$ in this expression. Note that in this case, the analytic extension $h(\cdot,\cdot)$
 satisfies
$$h(z,\bar w)=h(w)+(z-w)\pa h(w)+\dots,\qquad (z\to w).$$
This motivates the definition of the approximate Bergman kernel in  the case where $h$ is only a smooth function: we set
$$\ti K_w^\#(z)=K_w^\#(z)~e^{-2h_w(z)},$$
where
$$h_w(z):=h(w)+(z-w)\pa h(w).$$
 The approximate Bergman projection is defined accordingly:
$$\ti P_n^\#f(w)=(f, \ti K_w^\#)_{n\ti Q}.$$

\bs\no
The kernels $\ti K_n^\#(z,w)$ do not  have the Hermitian property. The important fact is that they are analytic in $z$.

\subsection*{Proof of Theorem \ref{ThmK}}
We shall prove the following estimate.

\begin{lemA1} \label{AppL1} If $z\in S$, $\delta(z)>2\delta_n$, and if $|z-w|<\delta_n$, then
$$\left|\ti K_w(z)-\ti K_w^\#(z)\right|\lesssim e^{nQ(z)}~e^{nQ(w)}.$$
\end{lemA1}

Before we prove the lemma, we use it to conclude the proof of Theorem \ref{ThmK}.
Recall that
$$\ti \K_n(z,w)=\ti K_w(z) ~e^{-n\ti Q(z)}~e^{-n\ti Q(w)}.$$
If we define
$$\ti \K_n^\#(z,w)=\ti K_w^\#(z) ~e^{-n\ti Q(z)}~e^{-n\ti Q(w)},$$
then
by Lemma A.1,
$$\ti \K_n(z,w)=\ti \K_n^\#(z,w)+O(1).$$
On the other hand, we have
$$\ti \K_n^\#(z,w)=\K_n^\#(z,w)~e^{h(z)+h(w)-2h_w(z)},$$
so
$$|\ti \K_n^\#(z,w)|=|\K_n^\#(z,w)|~(1+O(|w-z|^2)=|\K_n^\#(z,w)|+O(1).$$
It follows that
$$|\ti \K_n(z,w)|=|\K_n^\#(z,w)|+O(1),$$
as claimed in Theorem \ref{ThmK}. $\qed$

It remains to prove Lemma A.1.

\begin{lemA2} \label{AppL2} If $f$ is analytic and bounded in $D(z;2\delta_n)$ and $w\in D(z;\delta_n)$, then
$$\left|f(w)-\ti P^\#_n(\chi_z f)(w)\right|\lesssim \frac1{\sqrt n}~e^{nQ(w)}~\|f\|_{nQ}.$$\end{lemA2}

\ss\no Here $\chi=\chi_z$ is a cut-off function with $\chi=1$ in $D(z;3\delta_n/2)$ and $\chi=0$
outside $D(z;2\delta_n)$ satisfying  $\|\bp \chi\|_2\asymp 1.$

\ms\no \begin{proof} Wlog, $w=0$, so $\ti P^\#_n(\chi f)(w)$ is the integral
$$I^\#=\frac 1 \pi
\int \chi(\zeta)\cdot f(\zeta)\cdot2(\pa_1\pa_2 Q)(0,\bar\zeta)\cdot e^{2[h(\zeta)-h(0)-\bar \zeta\bp h(0)]}\cdot ne^{-2n[Q(\zeta,\bar \zeta)-Q(0,\bar \zeta)]}.$$
Since
$$\bp_\zeta\left[e^{-2n[Q(\zeta,\bar \zeta)-Q(0,\bar \zeta)]}\right]=-2[\pa_2Q(\zeta,\bar\zeta)-\pa_2Q(0,\bar\zeta)]~ne^{-2n[Q(\zeta,\bar \zeta)-Q(0,\bar \zeta)]},$$
we can rewrite the expression as follows:
$$I^\#=-\frac 1 \pi \int \frac1\zeta~f(\zeta)\chi(\zeta)A(\zeta)B(\zeta)~\bp\left[e^{-2n[Q(\zeta,\bar \zeta)-Q(0,\bar \zeta)]}\right],$$
where
$$A(\zeta)=\frac{\zeta~(\pa_1\pa_2 Q)(0,\bar\zeta)}{\pa_2Q(\zeta,\bar\zeta)-\pa_2Q(0,\bar\zeta)},$$
and
$$B(\zeta)=e^{2[h(\zeta)-h(0)-\bar \zeta\bp h(0)]}.$$
A trivial but  important observation is that
$$A,B=O(1)\quad ,\quad \bp A=O(|\zeta|)\quad ,\quad  \bp B=O(|\zeta|),$$
where the $O$-constants have uniform bounds throughout.

\bs\no Integrating by parts
we get
$$I^\#=f(0)+\epsilon_1+\epsilon_2,$$
where
$$\epsilon_1=\int \frac{f~(\bp\chi)~ AB}\zeta~e^{-2n[Q(\zeta,\bar \zeta)-Q(0,\bar \zeta)]}\quad,\quad \epsilon_2=\int \frac{f\chi~ \bp(AB)}\zeta~e^{-2n[Q(\zeta)-Q(0,\bar \zeta)]}.$$
Using that
$$|\epsilon_1|\lesssim \frac1{\delta_n}\int |f|~|\bp\chi|~e^{-2n[Q(\zeta)-\re Q(0,\bar \zeta)]}
\quad ,\quad
|\epsilon_2|\lesssim \int \chi ~|f|~e^{-2n[Q(\zeta)-\re Q(0,\bar \zeta)]},$$
and noting that Taylor's formula gives
$$ e^{-n[Q(\zeta)-2\re Q(0,\bar\zeta)]}\lesssim e^{nQ(0)-cn|\zeta|^2},\quad (c\sim \Delta Q(0)>0,\quad
|\zeta|\le 2\delta_n)$$
we find, by the Cauchy--Schwarz inequality, (since $|\zeta|\ge \delta_n$ when $\bp\chi(\zeta)\ne 0$)
$$|\epsilon_1|e^{-nQ(0)}\lesssim \frac {e^{-cn\delta_n^2}} {\delta_n}\|f\|_{nQ}\|\bp\chi\|_{L^2}\lesssim \frac 1 {\sqrt{n}}\|f\|_{nQ}$$
and
$$|\epsilon_2|e^{-nQ(0)}\lesssim \|f\|_{nQ}\left( \int e^{-nc|\zeta|^2}\right)^{1/2}\lesssim \frac 1 {\sqrt{n}}\|f\|_{nQ}.$$
The proof is finished.
\end{proof}

Suppose now that $\dist(z,{\mathbb C}\setminus S)\ge 2\delta_n$
and $|w-z|\le \delta_n$.

From Lemma A.2, we conclude that
\begin{equation}\label{E1}\left|\ti K_w(z)-\ti P_n\left[\chi_z \ti K_w^\#\right](z)\right|\lesssim e^{nQ(z)}~e^{nQ(w)}.\end{equation}
This is because
$$\ti P_n^\#\left[\chi_z \ti K_z\right](w)=(\chi_z \ti K_z, \ti K^\#_w)_{n\ti Q}=\overline{(\chi_z \ti K^\#_w, \ti K_z)}_{n\ti Q}=\overline{\ti P_n\left[\chi \ti K^\#_w\right](z)},$$
so
$$\left|\ti K_w(z)-\ti P_n\left[\chi_z \ti K_w^\#\right](z)\right|=\left|\ti K_z(w)-\ti P_n^\#\left[\chi_z \ti K_z\right](w)\right|
$$
and because (cf. \cite{AHM1}, Section 3)
$$\|\ti K_z\|=\sqrt{\ti K_z(z)}\lesssim \sqrt n e^{nQ(z)}.$$

\ms\no On the other hand, we will prove that
\begin{equation}\label{E2}\left|\ti K_z^\#(w)-\ti P_n\left[\chi_z \ti K_z^\#\right](w)\right|\lesssim e^{nQ(z)}~e^{nQ(w)},\end{equation}
which combined with \eqref{E1} proves Lemma A.1. The verification of the last inequality  is the same as in \cite{B1} or \cite{A},
depending on the observation that $L^2_{nQ}=L^2_{n\ti Q}$ with equivalence of norms. We give a detailed argument,  for completeness.

For given smooth $f$, consider $u$, the $L^2_{nQ}$-minimal solution to the problem
\begin{equation}\label{E3}\bar \partial u=\bar\partial f\quad \text{and}\quad u-f\in {\mathcal P}_{n-1}.\end{equation}
Since $\|u\|_{n\tilde Q}\lesssim \|u\|_{nQ}$, the $L^2_{n\tilde{Q}}$-minimal solution
$\tilde{u}$ to the problem \eqref{E3} satisfies $\|\tilde u\|_{n\tilde{Q}}\le C\|u\|_{nQ}$.
We next
observe that $P_n f$ is related to the $L^2_{nQ}$-minimal solution $u$ to the problem \eqref{E3}
by $u=f-P_n f$.

We write
$$u(\zeta)=\chi_z(\zeta)\ti K_w^\#(\zeta)-P_n\left[\chi_z\ti K_w^\#\right](\zeta),$$
i.e., $u$ is the $L^2_{nQ}$-minimal solution to \eqref{E3} for $f=\chi_z\cdot\ti K_w^\#$.
Let us verify that
\begin{equation}\label{E4}
\left\|u\right\|_{nQ}\lesssim \frac 1 {\sqrt{n}}\left\|\bar\partial\left(\chi_z\cdot \ti K_w^\#\right)\right\|_{nQ}.
\end{equation}
To prove this, we put
$$2\phi(\zeta)=2\check {Q}(\zeta)+n^{-1}\log\left(1+|\zeta|^2\right),$$
and consider the function $v_0$, the $L^2_{n\phi}$-minimal solution to the problem $\bar\partial v=\bar\partial\left(\chi_z\cdot \ti K_w^\#\right)$. Notice that $\phi$ is strictly subharmonic on $\C$.
By H\"{o}rmander's estimate (e.g. \cite{H}, p. 250)
$$\|v_0\|_{n\phi}^2\lesssim\int_{{\mathbb C}}
\left|\bar\partial\left(\chi_z\cdot\ti k_w^\#\right)\right|^2\frac {e^{-2n\phi}}
{n\Delta\phi}.$$
Since $\chi_z$ is supported in $S$, we hence have
$$\|v_0\|_{n\phi}\lesssim \frac 1 {\sqrt{n}}\left\|\bar\partial\left(\chi_z\cdot \ti K_w^\#\right)\right\|_{nQ}.$$

We next observe that by the growth assumption on $Q$ near infinity, we have an estimate $n\phi\le nQ+\text{const.}$ on ${\mathbb C}$, which gives $\|v_0\|_{nQ}\lesssim \|v_0\|_{n\phi}$. It yields that
$$\|v_0\|_{nQ}\lesssim \frac 1 {\sqrt{n}}\left\|\bar\partial\left(\chi_z\cdot \ti K_w^\#\right)\right\|_{nQ}.$$
But $v_0-\chi_z\cdot \ti K_w^\#$ belongs to the weighted Bergman space $A^2_{n\phi}$. Since $2n\phi(\zeta)=(n+1)\log|\zeta|^2+O(1)$ as $\zeta\to\infty$, the latter space coincides with ${\mathcal P}_{n-1}$ as sets. This shows that $v_0$ solves the problem \eqref{E3}. Since $\|u\|_{nQ}\le \|v_0\|_{nQ}$, we then obtain
\eqref{E4}.

By norms equivalence, \eqref{E4} implies that
\begin{equation}\label{E5}
\left\|\ti u\right\|_{n\ti{Q}}\lesssim \frac 1{\sqrt{n}}\left\|\bar\partial\left(\chi_z\cdot \ti K_w^\#\right)\right\|_{n\ti Q},
\end{equation}
where
$$\ti u=\chi_z\ti K_w^\#-\ti P_n\left[\chi_z\ti K_w^\#\right]$$ is the
$L^2_{n\ti Q}$-minimal solution to \eqref{E3} with $f=\chi_z\ti K_w^\#$.

We now set out to prove the pointwise estimate
\begin{equation}\label{E6}|\ti u(z)|\lesssim ne^{-cn\delta_n^2}e^{n(Q(z)+Q(w))}.\end{equation}
To prove this, we first observe that
$$\bar \partial \ti u(\zeta)=\bar\partial\left(\chi_z\cdot \ti K_w^\#\right)(\zeta)=\bar\partial \chi_z(\zeta)\cdot
\ti K_w^\#(\zeta),$$
whence, by the form of $\ti K_w^\#$ and Taylor's formula,
$$\left|\bar \partial \ti u(\zeta)\right|^2e^{-2nQ(\zeta)}\lesssim n^2\left|\bar\partial\chi_z(\zeta)\right|^2 e^{2n(Q(w)-c|\zeta-w|^2)}$$
with a positive constant $c\sim \Delta Q(z)$. Since $|\zeta-w|\ge \delta_n/2$ when $\bar\partial \chi(\zeta)\ne 0$, it yields
$$\left|\bar\partial\left(\chi_z\cdot \ti K_w^\#\right)\right|^2e^{-2nQ(\zeta)}
\lesssim n^2\left|\bar\partial \chi_z(\zeta)\right|^2e^{2nQ(w)-cn\delta_n^2}.$$
We have shown that
$$\left\|\bar\partial\left(\chi_z\cdot \ti K_w^\#\right)\right\|_{n Q}\lesssim ne^{-nc\delta_n^2}e^{n Q(w)}.$$
In view of the estimate \eqref{E5}, we then have
$$\|\ti u\|_{n Q}\lesssim \sqrt{n}e^{-nc\delta_n^2}e^{n Q(w)}.$$
Since $\ti u$ is analytic
in $D(z;1/\sqrt{n})$
we can now invoke the simple estimate (e.g. \cite{AHM1}, Lemma 3.2)
$$|\ti u(z)|^2e^{-2n Q(z)}\lesssim n\|\ti u\|_{n Q}^2$$
to get
$$|\ti u(z)|\lesssim ne^{-nc\delta_n^2}
e^{n( Q(z)+ Q(w))}.$$
This gives \eqref{E2}, and finishes the proof of Lemma A.1. $\qed$

\begin{rem} The corresponding estimate in \cite{A}, though correct, contains an unnecessary factor "$\sqrt{n}$''. \end{rem}


\begin{thebibliography}{999}
\bibitem{A} Ameur, Y., \textit{Near boundary asymptotics for correlation kernels}, J. Geom. Anal. (2011), available online at DOI 10.1007/s12220-011-9238-4.
\bibitem{AHM1} Ameur, Y., Hedenmalm, H., Makarov, N., \textit{Berezin transform in polynomial Bergman spaces}, Comm. Pure Appl. Math. \textbf{63} (2010), 1533--1584.
\bibitem{AHM2} Ameur, Y., Hedenmalm, H., Makarov, N., \textit{Fluctuations of eigenvalues of random normal matrices}, Duke Math. J. \textbf{159} (2011), 31--81.
\bibitem{AKM} Ameur, Y., Kang, N.-G., Makarov, N., In preparation.
\bibitem{B1} Berman, R., \textit{Bergman kernels and weighted equilibrium
measures of $\C^n$.} Indiana Univ. J. Math. \textbf{5} (2009).
\bibitem{B2} Berman, R., \textit{Determinantal point processes and fermions on complex manifolds: bulk universality}, Preprint in 2008 at arXiv.org/abs/math.CV/08113341.
\bibitem{B} Borodin, A., \textit{Determinantal point processes}, Preprint in 2009 at ArXiv.org/0911.1153.
\bibitem{BS} Borodin, A., Sinclair, C. D., \textit{The Ginibre ensemble of real random matrices and its scaling limits}, Commun. Math. Phys. \textbf{291} (2009), 177--224.
\bibitem{EF} Elbau, P., Felder, G., \textit{Density of eigenvalues of random normal matrices}, Commun. Math. Phys. \textbf{259}
(2005), 433--450.
\bibitem{HM} Hedenmalm, H., Makarov, N., \textit{Coulomb gas ensembles and Laplacian growth}, Preprint in 2011 at arXiv.org/abs/math.PR/1106.2971.
\bibitem{HS} Hedenmalm, H., Shimorin, S., \textit{Hele-Shaw flow on hyperbolic surfaces}, J. Math. Pures. Appl. \textbf{81} (2002), 187--222.
\bibitem{H} H\"ormander, L., \textit{Notions of convexity},
Birkh\"auser 1994.
\bibitem{J} Johansson, K., \textit{On fluctuations of eigenvalues of
random Hermitian matrices}, Duke Math. J. \textbf{91} (1998),
151--204.
\bibitem{M} Mehta, M. L., \textit{Random matrices}, Academic Press 2004.
\bibitem{RV} Rider, B., Vir\'{a}g, B., \textit{The noise in the circular law
and the Gaussian free field}, Internat. Math. Research notices
\textbf{2007}, no. 2.
\bibitem{ST} Saff, E. B., Totik, V., \textit{Logarithmic potentials with external fields}, Springer 1997.
\bibitem{Sa} Sakai, M., \textit{Regularity of a boundary having a Schwarz function}, Acta Math. \textbf{166} (1991), 263--297.
\bibitem{S} Soshnikov, A., \textit{Determinantal random point fields},
Russ. Math. Surv. \textbf{55} (2000), 923--975.
\bibitem{WZ1} Wiegmann, P., Zabrodin, A., \textit{Large $N$ expansion for the $2D$ Dyson gas}, J. Phys. A.: Math. Gen. \textbf{39}
(2006), 8933--8964.
\bibitem{WZ2} Wiegmann, P., Zabrodin, A., \textit{Large $N$ expansion for the normal and complex matrix ensembles}, Frontiers in number theory, physics, and geometry I
(2006), Part I, 213--229.
\bibitem{WZ3} Wiegmann, P., Zabrodin, A., \textit{Large scale correlations in normal non-Hermitian matrix ensembles}, J. Phys. A.:
Math. Gen. \textbf{36} (2003), 3411--3424.
\bibitem{Z} Zabrodin, A., \textit{Matrix models and growth
processes: from viscous flows to the quantum Hall effect}, Preprint
in 2004 at arXiv.org/abs/hep-th/0411437.
\end{thebibliography}
\end{document}